\documentclass[12pt]{article}
\usepackage{amsmath,amsthm,amsfonts,latexsym,amscd,amssymb,bm}

 \numberwithin{equation}{section}
\newtheorem{theorem}{Theorem}[section]
\newtheorem{lemma}[theorem]{Lemma}

\newtheorem{proposition}[theorem]{Proposition}

\theoremstyle{definition}

\newtheorem{example}[theorem]{Example}

\theoremstyle{remark}

\newcommand{\n}{\noindent}
\newcommand{\set}[1]{\mathcal{#1}}

\begin{document}
	\title{Isotopic classes of transversals in Dihedral group $ D_{2n} $, $n$ odd.}
	\author{Surendra Kumar Mishra \\
	Department of Mathematics, University of Allahabad\\
	Allahabad (India) 211002 \\
	Mobile Number : 7860130698\\
	 \textbf{Email:} surendramishra557@gmail.com  \and
	 R. P. Shukla \\
	 Department of Mathematics, University of Allahabad\\
	 Allahabad (India) 211002 \\
	  \textbf{Email:} shuklarp@gmail.com}
\date{}	
 \maketitle
 
 \begin{abstract} 
 \n	In this article we determine the number of isotopic classes of transversals of a subgroup of 
order 2 in $ D_{2n}$ ($n$ is a positive  odd integer greater than 1), where isotopism classes are 
formed with respect to the induced right loop structures. We also determine the cyclic index of the 
Affine group Aff$(1,p^{2})$, where $p$ is an odd prime.\end{abstract} 

 
 \n  \textbf{\textit{2010 Mathematics Subject Classification:}} 20D60, 20N05.\\

  \section{Introduction}\label{s1} 
 Let $ G $  be a finite group and $ H $  be a subgroup of  $ G $. A normalized right transversal(NRT) 
$ T $ of $ H $ in $ G $ is a collection of elements of $G$ obtained by selecting one and only one 
element from each right coset of $H$ in $ G $ and  $ 1 \in T $. Let $T$ be an NRT of $H$ in $G$. 
Then $ T $ has a binary operation $\circ $ induced by the binary operation of $G$,  given by 
$ \{x\circ y\}= Hxy\cap T $ with respect to which $T$ becomes a right loop with identity 1, 
that is a right quasigroup with both sided identity (see \cite[Proposition 4.3.3, p.102]{jdh} 
and \cite{rls}). Conversely, every right loop can be embedded as an NRT in a group with some 
universal property (see \cite[Theorem $3.4$, p.74]{rls}).\\
 
 Let $(L_{1}, \circ_{1}) $ and $(L_{2},\circ_{2}) $ be two groupoids. We say that $L_{1}$ is 
$\mathit{isotopic}$ to  $L_{2}$ if there are bijective maps $f$, $g$ and $h$ from $L_{1}$ to 
$L_{2}$ such that $f(a)\circ_{2}g(b) = h(a\circ_{1}b) $ for all $a, b \in L_1$. This type of 
triple $(f,g,h) $ is known as an $\mathit{isotopism}$ or an  $ \mathit{isotopy} $ from $L_{1}$ 
to $L_{2}$. An isotopy $ (f,f,f) $ from  $L_{1}$ to $L_{2}$ is known as an $\mathit{isomorphism}$. 
We say that $(L_{1}, \circ_{1}) $ and $(L_{1},\circ_{2}) $ are principal isotopic if $ (f,g, I) $ 
is an isotopy between $(L_{1}, \circ_{1}) $ and $ (L_{1},\circ_{2})$, where $I$ is the identity 
map on $L_{1}$ (see\cite[p.248]{brk}). Let $ \mathcal{T}(G, H) $ denote the set of all NRTs to $H$ 
in $G$. Let $T_{1}, T_{2} \in\mathcal{T}(G,H) $. If the induced right loop structures in $T_{1}$ 
and $T_{2}$ are isotopic, then we say that $ T_{1}$ and $ T_{2} $ are isotopic.\\
 
 Let $ (S, \circ) $ be a right loop. For $ a \in S $, define $L_{a}^{\circ}:S \rightarrow S $ 
by $ L_{a}^{\circ}(x) = a\circ x $ and $ R_{a}^{\circ}: S \rightarrow S $ by 
$ R_{a}^{\circ}(x)  =x \circ a $. An element $\alpha \in S $ is said to be a {\textit {left nonsingular}}  
if $ L_{\alpha}^{\circ} $ is a bijection of $ S $. It is observed in \cite[Theorem 1 A, p.249]{brk} 
that if  $ \alpha\in S $ is left nonsingular and if $ \beta\in S, $ then there is a principal 
isotopy $((R_{\beta}^{\circ})^{-1},(L_{\alpha}^{\circ})^{-1}, I) $ between  
$( S, \circ_{\alpha,\beta}) $ and $ (S, \circ)$, where for all $a, b\in S 
$,  $ a\circ_{\alpha,\beta}b = (R_{\beta}^{\circ})^{-1}(a)\circ(L_{\alpha}^{\circ})^{-1}(b) $ 
and  every principal isotope of $(S, \circ)$ is of this form. We denote this isotope of $ S $ 
by   $ S_{\alpha,\beta}  $. It is easy to observed that the identity of  
$(S, \circ_{\alpha, \beta)} $ is $ \alpha \circ \beta $ . It is also observed in  
\cite[Lemma 1 A, p.248]{brk} that if a right loop $ (L_{1},\circ_{1}) $ is isotopic to the 
right loop $ (L_{2}, \circ_{2})$, then $ (L_{2}, \circ_{2}) $ is isomorphic to a principal 
isotope of $(L_{1}, \circ_{1}) $.\\
  
    We denote the set of all isotopism classes of  elements in  $ \mathcal{T}(G,H) $ by  
		$\mathcal{I}tp(G,H) $. In Section $2$, we describe the number of elements in 
		$\mathcal{I}tp(D_{2n},H) $, where $n$ is an odd integer greater than $1$ and $H$ is 
		a  subgroup of order $2$ of the dihedral group $D_{2n}$ of order $2n$. The results proved 
		are  generalizations of results in the Section $4$ of \cite{vk}. In Section $3$, 
		we compute the cyclic index of the one dimensional affine group Aff$(1,p^{2})$ 
		of $\mathbb{Z}_{p^{2}} $, where $p$ is an odd prime .

   	\section{ Isotopic classes of transversals}\label{s2}
   	 
   Let $n\in \mathbb{N}$. Consider $\mathbb{Z}_n =\{0, 1,2,\cdots, n-1 \}$ the ring of integers modulo $n$. 
    Let $A\subseteq\mathbb{Z}_{n}\setminus\{0\} $. Let $a$, $b$ $\in\mathbb{Z}_{n} $ and define an 
		operation $\circ_A$ on  $\mathbb{Z}_{n}$ as
   \begin{equation}\label{s2e1} a\circ_A b = \left\{ 
    \begin{array}{l l}
    	a+b & \quad \mbox{ if $ b\notin A $ } \\
    	b-a & \quad \mbox{ if $ b\in A $ }.\\
    \end{array} \right. 
     \end{equation} \\
           It can be easily verified that $ (\mathbb{Z}_{n}, \circ_A)  $ is a right loop
					(see \cite[Section 4]{vk}). We denote this right loop by $ \mathbb{Z}_{n}^{A} $. 
					We observed that if  $ A=\phi,$ then  $ \mathbb{Z}_{n}^{A} $ is the additive  group $ \mathbb{Z}_{n} $.\\

	 Throughout this section, let $ G = D_{2n}= \langle a, b : a^{2}=b^{n}=1,~ aba=b^{-1} \rangle $ 
	be the dihedral group of order $2n$, where $n$ is an odd integer greater than $1$. Let 
	$ H=\{1, x \} $ be a subgroup of $G$ of order $2$ and $ K=\langle b \rangle $ be the cyclic 
	subgroup of $G$ of order $n$. Let $ \sigma : K\rightarrow H $ be a function with 
	$ \sigma (1)=1 $. Then $ T_{\sigma}= \{\sigma (b^{j})b^{j}:1\leq j \leq n\} \in \mathcal{T} (G, H)$ 
	and all NRTs of $H$ in $G$ are of this form. Let $A=\{j\in \mathbb{Z}_{n}| \sigma(b^{j})=x\}$. 
	Since $A$ is completely determined by $\sigma$, we denote $ T_{\sigma}$ by $ T_{A} $. Clearly, 
	the map  $\sigma(b^{j})b^j \rightarrow j$ from $ T_{\sigma}$ to $ \mathbb{Z}_{n}^{A} $ is an 
	isomorphism of right loops. So we may identify the right loop $ T_{A} $ with the right loop 
	$\mathbb{Z}_{n}^{A} $ by means of the above isomorphism. We observe that $ T_{\phi} = K \cong \mathbb{Z}_{n} $.\\

	  Let $ U_{n} $ be the group of units of  the ring $ \mathbb{Z}_{n}  $. 
		For $\nu\in U_n,~u\in\mathbb{Z}_n,$ let $f_{\nu,u}:\mathbb{Z}_{n}\rightarrow \mathbb{Z}_{n} $ 
		be the affine map defined by $f_{\nu, u}(x) =\nu x + u $. 
		For $ B(\neq \phi)\subseteq \mathbb{Z}_{n}\setminus\{0\} $, 
		define $ \set{\chi}_{_B} = \{f^{-1}_{\lambda, t^{\prime}}(B): t^{\prime}\notin B,~ \lambda\in U_n \}\cup\{(f^{-1}_{\lambda, t^{\prime}}(B))^{\prime}: t^{\prime}\in B,~ \lambda\in U_n\}$, 
		where for any subset $A$ of $\mathbb{Z}_{n} $, $A^{\prime}   = \mathbb{Z}_{n} \setminus A  $. 
		If  $ B=\phi, $ then we define $ \chi_{_B}=\phi.$ Now we have the following theorem which is  
		a generalization of  \cite [Theorem  $4.7$]{vk}. 	  
	  
	 \begin{theorem}\label{s2t1} Let $ L=T_{A} \in \mathcal{T}(G, H) $. Then   $ S\in \mathcal{T}(G, H)  $ 
	is isotopic to $ L $ if and only if $ S= T_{C}$ for some  $ C \in \chi_{_A} $.  \end{theorem}
	  
	 \begin{proof} By descriptions in the second paragraph of this section and  third paragraph of 
	the Section $1$, it is enough to prove that for a subset $C$ of $\mathbb{Z}_n$, $\mathbb{Z}_{n}^{C} $ 
	is isomorphic to a principal isotope of $\mathbb{Z}_{n}^{A}$ if and only if $C\in\chi_{_A}.$ \\ 
	 	 	  Assume that $ A=\phi $. Then   $ \mathbb{Z}_{n}^{A}  = \mathbb{Z}_{n} $. Since 
				$ \mathcal{T}(G,H) $ contains exactly one loop transversal (see \cite[ Corollary $4.5$]{vk}) 
				and a right loop isotopic to a loop is itself a loop (see \cite [Corollary $3.2$]{vk}), 
				we are done in this case. Next, assume that $ A \neq \phi $.\\
	          
	       Let  $ \beta\in\mathbb{Z}_{n}\setminus\{0\} $. Consider the  bijective 
				maps $ \psi_{\beta} $ and  				$\rho_{\beta} $ on $\mathbb{Z}_{n}$ defined by 
				$ \psi_{\beta}(x) = x+\beta $ and $\rho_{\beta}(x)  = \beta - x $. Then by (\ref{s2e1}),  
	           \begin{equation} \label{s2e2} R_{\beta}^{\circ_{A}}  = \left\{
	           \begin{array}{l l}
	           \psi_{\beta} & \quad \mbox{if $ \beta \notin A$}\\
	            \rho_{\beta} & \quad \mbox{ if $ \beta \in A $ }.\\
	           \end{array} \right. \end{equation}
	           
	           Let $ \alpha\in \mathbb{Z}_{n}^{A} $ be a left nonsingular element. 
						Consider the principal isotope $(\mathbb{Z}_n^{A})_{\alpha,\beta} $ of 
						$\mathbb{Z}_{n}^{A}$. Let $\circ_{\alpha,\beta}$ denote the binary operation of 
						$(\mathbb{Z}_{n}^{A})_{\alpha, \beta}. $
	           First let  $ \beta\notin A  $. Let  $ u,v \in (\mathbb{Z}_n^{A})_{\alpha,\beta}$. 
						Since $\beta \notin A,$ by (\ref{s2e2}),
	                 $(R_{\beta}^{\circ_{A}})^{-1}(u)= \psi_{\beta}^{-1}(u)= u-\beta$. Further,  
									since $ v\notin A $ if and only if $v+k\alpha\notin A ~  (k\in\mathbb{Z})$ 
									(see \cite[Lemma $4.1$]{vk}), \[ (L_{\alpha}^{\circ_{A}})^{-1}(v) = \left\{
	           \begin{array}{l l}
	           v-\alpha & \quad \mbox{if $ v\notin A$}\\
	           v+\alpha & \quad \mbox{ if $ v\in A$ }.\\
	           \end{array} \right. \]
	           Therefore,
	           
	         $ ~~~~~~~~~~~~~~~~$ $ u\circ_{\alpha,\beta} v= (R_{\beta}^{\circ_{A}})^{-1}(u)\circ_{A}(L_{\alpha}^{\circ_{A}})^{-1}(v) $
				 			
	         \begin{equation*}~~~~~~~~~~~~  = \left\{
	           \begin{array}{l l}
	           (u-\beta)\circ_{A}(v-\alpha)    & \quad \mbox{if $ v\notin A$}\\
	           (u-\beta)\circ_{A}(v+\alpha)	& \quad \mbox{ if $v \in A$}\\
	           \end{array} \right. \end{equation*}
	            \begin{equation}\label{s2e3}  ~~~~~~~~~~~~ = \left\{
	           \begin{array}{l l}
	           (v+u)-(\beta+\alpha)    & \quad \mbox{if $ v\notin A$}\\
	           (v-u)+(\beta+\alpha)	& \quad \mbox{ if $v \in A$}.\\
	           \end{array} \right.  \end{equation}   Let $\nu\in U_{n}$. Then the binary operation  
						$ \circ_{\alpha,\beta} $  and the map $ f_{\nu,(\beta+\alpha)} $ defines a binary operation 
						$ \circ_{f_{\nu,(\beta+\alpha)}} $ on  $ \mathbb{Z}_{n} $ so that $ f_{\nu,(\beta+\alpha)} $ 
						is an isomorphism of right loops from $(\mathbb{Z}_{n},\circ_{f_{\nu,\beta+\alpha}})$ to 
						$ ((\mathbb{Z}_n^{A})_{\alpha,\beta}, \circ_{\alpha,\beta}) $. 
						Thus \begin{equation*}  u\circ_{f_{\nu,(\beta+\alpha)} }v=f^{-1}_{\nu,(\beta+\alpha)}
						\big (f_{\nu,(\beta+\alpha)}(u)\circ_{\alpha,\beta}f_{\nu,(\beta+\alpha)}(v)\big ) \end{equation*}

	           \begin{equation}\label{s2e4} ~~~~~~~~~~~~~~ = \left\{
	           \begin{array}{l l}
	           (v+u)   & \quad \mbox{if $ v\notin f^{-1}_{\nu,(\beta+\alpha)}(A)$}\\
	           (v-u)	& \quad \mbox{ if $v\in f^{-1}_{\nu,(\beta+\alpha)}(A)$ }.\\
	           \end{array} \right. \end{equation}  This implies that the right loop $(\mathbb{Z}_{n}$, $\circ_{f_{\nu,(\beta+\alpha)}} $) is $(\mathbb{Z}_{n}^{C}, \circ_C )$, where $C= f^{-1}_{\nu,(\beta+\alpha)}(A)$$\in\chi_{_A } $. \\
	           
	           Let  $ \beta\in A $. Since $ v\in A $ if and only if $v+k\alpha \in A~ (k\in\mathbb{Z})$ (see \cite[Lemma $4.1$]{vk}),  
					\begin{equation*}
				u\circ_{\alpha,\beta}v= R^{-1}_{\beta}(u)\circ_{A}L_{\alpha}^{-1}(v)
					\end{equation*}	 
						 \[~~~~~~~~~~~~~~~~~~~~~~~~~~~~~~~~~~ = \left\{
	           \begin{array}{l l}
	           (\beta-u)\circ_{A}(v-\alpha)    & \quad \mbox{if $ v\notin A$}\\
	           (\beta-u)\circ_{A}(v+\alpha)	& \quad \mbox{ if $v \in A$ }\\
	           \end{array} \right.\]
						\begin{equation}\label{s2e5} ~~~~~~~~~~~~~~~~~~~~~~~~~~~~~~~~~~~~~~~
	           = \left\{
	           \begin{array}{l l}
	           (v-u)+(\beta-\alpha) & \quad \mbox{if $ v\notin A$}\\
	           (v+u)-(\beta-\alpha)	& \quad \mbox{ if $v \in A$ }.\\
	           \end{array} \right. \end{equation}
	            The  arguments used for the case $\beta\notin A $ imply that  the map  $f_{\nu,(\beta-\alpha)} $ is an isomorphism of right loops from $\mathbb{Z}_{n}^{C} $ to $(\mathbb{Z}_{n}^{A})_{\alpha,\beta} $, where $C= (f^{-1}_{\nu,\beta-\alpha})^{\prime}\in\chi_{_A}$. Thus $\mathbb{Z}_{n}^{C}$ is isotopic to $ \mathbb{Z}_{n}^{A} $ if $ C\in\chi_{_A} $.\\
	             
	            Conversely, let C be a subset of $ \mathbb{Z}_{n}\setminus\{0\}$ such that  
							$ \mathbb{Z}_{n}^{C} $ is isotopic to $ \mathbb{Z}_{n}^{A} $. Let $ (f,g,h) $  
							be an isotopy  from $ \mathbb{Z}_{n}^{C} $ to $ \mathbb{Z}_{n}^{A} $. Then as described in  
							the third paragraph of the Section \ref{s1}, $ (f,g,h)=  ((R_{\beta}^{o_{A}})^{-1}, (L_{\alpha}^{o_{A}} )^{-1}, I )(h,h,h) $, where $ h $  is an isomorphism from  $ \mathbb{Z}_{n}^{C} $  to a principal isotope $ L_{1}=(\mathbb{Z}_{n}^{A})_{\alpha,\beta}$ of $\mathbb{Z}_{n}^{\circ_A}$ for some $\beta\in\mathbb{Z}_{n}$ and some  left nonsingular element $\alpha$ in $ \mathbb{Z}_{n}^{A} $. As we have observed in (\ref{s2e2}) if $\beta\notin A$, then $ R_{\beta}^{\circ_{A}}= \psi_{\beta}  $ and if $\beta\in A $, then $R_{\beta}^{\circ_{A}}= \rho_{\beta} $. 
							
	            Let $v\in \mathbb{Z}_{n}^{C}$.  Since $h$ is an isomorphim from $\mathbb{Z}_{n}^{C} $ to $ L_{1} $,
	            \begin{equation}\label{s2e6}
	              R_{v}^{\circ_{C}}= h^{-1}R^{\circ_{\alpha,\beta}}_{h(v)}h, 
	             \end{equation} where $ \circ_{\alpha,\beta} $ is the binary operation of $(\mathbb{Z}_{n}^{A})_{\alpha,\beta}$.\\ 
	           
	           Assume that  $ \beta \notin A $. By $(\ref{s2e3})$, we have 
	           \begin{equation}\label{s2e7}
	         R_{h(v)}^{\circ_{\alpha,\beta}}=\left\{ \begin{array}{l l}
	         \psi_{_{h(v)-(\beta+ \alpha)}} & \quad \mbox{ if $ h(v)\notin A $ } \\
	         \rho_{_{h(v)+(\beta+\alpha)}} & \quad \mbox{ if $h(v)\in A $ }.\\
	         \end{array} \right.
	           \end{equation} 
	          
	          Since for any $x\in\mathbb{Z}_{n}$, $\rho_{_x}$ is of order $2$, $\psi_{x} $ is of odd order in 
						the symmetric group Sym$(\mathbb{Z}_n)$ of $\mathbb{Z}_n$, and the conjugate elements in a group have 
						same order, therefore from $(\ref{s2e6})$ and $(\ref{s2e7})$,  $h^{-1}\psi_{_{h(v)-(\beta+ \alpha)}}h=\psi_{v} $ if $ h(v) \notin A $ and $h^{-1}\rho_{_{h(v)+(\beta+ \alpha)}}h = \rho_{v} $ if $ h(v)\in A $.\\
	           
	                Assume that $ h(v)\notin A $. Then by $(\ref{s2e7})$ and as remarked above,  
									$ \big (h^{-1}\psi_{_{h(v)-(\beta+\alpha)}}\\h\big )(u) = \psi_{v}(u) $, that is  $ h(u+v) = h(u)+h(v)-(\beta+\alpha)$ for all $u\in\mathbb{Z}_n$.
	                 This implies,  $ h(0) = (\beta+\alpha) $ and  by induction we obtain that for each $u\in\mathbb{Z}_n$, \begin{equation}\label{s2e8}
	                h(u)= (h(1)-h(0))u + h(0),
	                \end{equation}  that is  $ h(u) = \nu u+ t $, where $ t= h(0)=\beta + \alpha $ and  $\nu = (h(1)-h(0)) $. Since $ h $ is a bijection on $\mathbb{Z}_n$, $ \nu \in U_{n} $ and so $h = f_{\nu,t} $.\\
	                
	                 Next, assume that $ h(v)\in A.$ Then by $(\ref{s2e7})$ and as remarked earlier, $ \big (h^{-1}\rho_{_{h(v)+(\beta+\alpha)}}h\big )(u) = \rho_{_v}(u)$ for all $u\in\mathbb{Z}_n$. Thus $ h(v-u)= h(v)-h(u)+ (\beta+\alpha)$ and so  $ h(u+v)= h(v+u)=h(v)-h(-u)+(\beta+\alpha)$ for all $u\in \mathbb{Z}_{n}$. Again we note that $ h(0)=\beta+\alpha $. Thus by induction for each $u\in\mathbb{Z}_n$, \begin{equation}\label{s2e9}  h(u)= (h(1)-h(0))u + h(0). \end{equation} Since $h$ is a bijection on $\mathbb{Z}_{n}$, $ (h(1)-h(0)) \in U_{n}$.\\
	                 
	                Thus in both  cases, $h=f_{\nu,t} $, where  $ t= h(0) =(\beta+\alpha) $ and $ \nu = \big (h(1)- h(0)\big )\in U_{n} $.
	                 
	                Hence we have, 
	               $ u\circ_{C}v = h^{-1}(h(u)\circ_{\alpha,\beta}h(v)) = f^{-1}_{\nu,(\beta+\alpha)}(f_{\nu,(\beta+\alpha)}(u)\circ_{\alpha,\beta}f_{\nu,(\beta+\alpha)}(v))$ for all $u, v \in\mathbb{Z}_{n} $.
	                 Thus as observed in $(\ref{s2e4})$, we have \begin{equation}\label{s2e10}
	                   u\circ_{C}v
	                = \left\{
	                \begin{array}{l l}
	                (v+u)   &  \quad \mbox{ if $ v\notin f^{-1}_{\nu,(\beta+\alpha)}(A)$}\\
	                (v-u)	& \quad \mbox{ if $v\in f^{-1}_{\nu,(\beta+\alpha)}(A)$} \\
	                \end{array} \right. \end{equation}\\ 
	                for all $u, v \in\mathbb{Z}_{n} $.
	                 This implies $ C= f^{-1}_{\nu,(\beta+\alpha)}(A) $.\\
	                 
	                  Next, assume that  $ \beta \in A $.
	                 Let $ u, v\in L_{1} $. Then by $(\ref{s2e5})$,    
	                 \[ u\circ_{\alpha,\beta}v
	           = \left\{
	           \begin{array}{l l}
	           (\beta-u)\circ_{A}(v-\alpha)    & \quad \mbox{if $ v\notin A$}\\
	           (\beta-u)\circ_{A}(v+\alpha)	& \quad \mbox{ if $v \in A$ }\\
	           \end{array} \right.\]
									\begin{equation}\label{s2e11} ~~~~~~~~~~~~
	                 = \left\{
	                 \begin{array}{l l}
	                 (v-u)+(\beta-\alpha) & \quad \mbox{if $ v\notin A$}\\
	                 (v+u)-(\beta-\alpha)	& \quad \mbox{ if $v \in A$ }.\\
	                 \end{array} \right. \end{equation}  Since $ h $ is an isomorphism from  $ \mathbb{Z}_{n}^{C} $ to  $ L_{1}$,  we have 
	                 \begin{equation}\label{s2e12}  R_{v}^{\circ_{C}} =
	                 h^{-1} R_{h(v)}^{\circ_{\alpha,\beta}}h \end{equation} and by $(\ref{s2e5})$,\begin{equation}\label{s2e13}
	                 R_{h(v)}^{\circ_{\alpha,\beta}}=\left\{ \begin{array}{l l}
	                 	\psi_{_{h(v)-(\beta- \alpha)}} & \quad \mbox{ if $ h(v)\in A $ } \\
	                 	\rho_{_{h(v)+(\beta-\alpha)}} & \quad \mbox{ if $h(v)\notin A $ }.\\
	                 \end{array} \right. \end{equation}
	                 
	                  Thus as argued for the case $\beta\notin A$, we have, $h= f_{\nu,t}$, where  
	                  $\nu= \big (h(1)-h(0)\big )\in U_{n},~ t=\beta-\alpha $ and

	                  \begin{equation}\label{s2e14}  u\circ_{C}v = \left\{
	                \begin{array}{l l}
	                	(v+u) & \quad \mbox{ if $ v\in f^{-1}_{\nu,(\beta-\alpha)}(A)$}\\
	                	(v-u)	& \quad \mbox{ if $v \notin f^{-1}_{\nu,(\beta-\alpha)}(A)$ }.\\
	                \end{array} \right. \end{equation}  Thus $ C = \mathbb{Z}_{n}\setminus (f^{-1}_{\nu,(\beta-\alpha)}(A)). $ This completes the proof of the theorem. 
	                \end{proof}
	                
	                Let $ \mathcal{S} $ denote a permutation group on a finite set $ S.$ Let $ \mid S \mid=k $. For $ \tau\in\mathcal{S} $, let $ c_{l} (\sigma ) $ denote the number of $l$-cycles in the disjoint cycle decomposition of  $ \tau. $ Let $ \mathbb{Q}[x_{1},\cdots,x_{k}] $ denote the polynomial ring in indeterminates $ x_{1},\cdots, x_{k} $. Then the cyclic index  $ P_{\mathcal{S}}(x_{1},\cdots, x_{k}) $ of $ \mathcal{S} $ is defined to be \begin{equation*}
	                P_{\mathcal{S}}(x_{1},\cdots,x_{k}) =\frac{1}{\mid \mathcal{S} \mid} \sum_{\tau\in\mathcal{S}} x_{1}^{c_{1}(\tau)} \ldots  x_{k}^{c_{k}(\tau)} 
	                \end{equation*} (see\cite[p.146]{bjn}).\\
	                
	               Now we state the following theorem whose proof  may be imitated from the proof of the Theorem $4.9$ of \cite{vk} replacing \cite[Theorem $4.7$]{vk} by Theorem $2.1$, $\mathbb{Z}_{p} $ by $\mathbb{Z}_{n}$, and Aff$(1,p)$ by Aff$(1,n)$.
	                \begin{theorem}\label{s2t2}  Let $n$ be an odd positive integer greater than 1, $G= D_{2n}$ and H be a subgroup of $G$ of order $2$. Then  \begin{center}
	                		$\mathcal{I}tp(G, H) = \frac{\text{P}_{\text{Aff}(1, n)}(2,~ 2,\cdots ,2)}{2}.$ \end{center}
	                \end{theorem}
	                
	                \section{Cyclic index of $ \text{Aff}(1, p^{2}) $}\label{s3}
	                 In this section we determine the cyclic index of  the one dimensional affine group $ \text{Aff}(1, p^{2}) $ of  $\mathbb{Z}_{p^{2}}$ ($p$ is an odd prime).\\
	                 
	                  Let $ J =\langle p \rangle =\{0,p,2p,\cdots,(p-1)p\} $ denote the unique maximal ideal in $\mathbb{Z}_{p^{2}}$.
	                 \begin{lemma}\label{s3l1} Let $ \nu\in U_{p^{2}}\setminus 1+J $  and $x\in\mathbb{Z}_{p^2}$. Then  $\nu x\equiv x (\text{mod}~p^{2}) $ if and only if $ x \equiv 0 (\text{mod}~ p^{2}) $.
	                 \end{lemma}
	                 	  \begin{proof} Assume that $ \nu x\equiv x (mod~p^{2}) $. Then \begin{equation}
	                  (\nu-1)x\equiv 0 (mod~p^{2}).      
	                \end{equation} Since $ (\nu-1,~ p^{2} )= 1 $,  $(3.1)$ is satisfied for $x\equiv0$ (mod$~p^{2} $) only. Obviously $x=0$ satisfies $\nu
	                x\equiv x (\text{mod}~p^{2}) $.\end{proof}  
									
            \begin{lemma}\label{s3l2}
            Let $\nu= 1+kp $ and $ t^{\prime} =lp\in J $, where $k(\neq 0),~ l \in \{0,1, 2,\cdots , p-1\}$. Let $x \in \mathbb{Z}_{p^2} $. Then $\nu x + t^{\prime} \equiv x (\text{mod}~p^2) $ if and only if $x\in -k^{\prime}l+J$, where $k^{\prime} k\equiv 1(\text{mod}~p). $
            \begin{proof} Since $p$ is a prime, there exist a unique integer $k^{\prime}(\text{modulo}~ p)$ such that $k^{\prime}k\equiv  1(\text{mod}~ p).$ Now\\  
             $\nu x + t^{\prime} \equiv x (\text{mod}~p^2)  \Leftrightarrow   
            (\nu-1) x + lp  \equiv 0 (\text{mod}~p^2) \Leftrightarrow 
            (kx + l)p  \equiv 0 (\text{mod} ~ p^2) \Leftrightarrow
           kx + l \equiv 0 (\text{mod}~p) \Leftrightarrow x\equiv -k^{\prime}l(\text{mod}~p).$
                        \end{proof} 
            \end{lemma} 
	                \begin{proposition}\label{s3p1} The cyclic index of the affine group $ P_{\text{Aff}(1, p^{2})} $ is \begin{equation*} \begin{split}
	                	\text{P}_{\text{Aff}(1, p^{2})}(x_{1},\cdots ,x_{p^{2}})=&\frac{1}{p^{2}\varphi(p^{2})}[x_{1}^{p^{2}}+(p-1)x_{p}^{p} + p^{2}\sum_{t|p-1, t\neq 1}\varphi(t)x_{1}x_{t}^{\frac{p^{2}-1}{t}}+\\&
	                	p^{2} \sum_{t|p-1, t\neq 1}\varphi(tp)x_{1} x_{tp}^{\frac{p-1}{t}}x_{t}^{\frac{p-1}{t}}  + p (p-1)x_{1}^{p}x_{p}^{p-1}+ p\varphi(p^{2})x_{p^{2}}], \end{split}     
	                \end{equation*}\end{proposition} where $ \varphi $ is the Euler's phi-function.\\ \begin{proof}
	                 Consider the  following pairwise disjoint subsets of  $ \text{Aff}(1, p^{2}) $ covering it :
	                
									\begin{enumerate}	                
	                \item  $S_{0}=\{I= \text {the identity permutation}\} $
	                \item  $ S_{1}=\{f_{1,t^{\prime}} : t^{\prime} \in J\setminus\{0\} \}$ 
	                \item  $S_{2} = \{f_{\nu,t^{\prime}} : \nu\in U_{p^{2}}\setminus 1+J ~ \text{and} ~ t^{\prime}\in\mathbb{Z}_{p^{2}}\} $ 
	                \item  $ S_{3}=\{f_{\nu,t^{\prime}} : \nu(\neq1)\in1+J ~\text{and}~ t^{\prime}\in J \} $
	                \item $S_{4}=\{f_{\nu,t^{\prime}} :  \nu\in 1+J ~\text{and}~ t^{\prime}\in U_{p^{2}}\}. $
	                \end{enumerate}

	                We note that, $I= f_{1,0} $  is the product of $p^2$ disjoint 1-cycles.
	                Let  $ f_{1,t^{\prime}}\in  S_{1}  $. Since $t^{\prime}\neq 0 $ and  $t^{\prime} \equiv 0 (\text{mod}~p)$, $ f_{1,t^{\prime}}$ is the product of $ p  $ disjoint $p$-cycles.\\ 
	                
	                                  By \cite[Lemma 2]{ftr}, $f_{\nu, t^{\prime}}\in S_{2} $ has same cycle type as the cycle type of $f_{\nu,0}. $ Since $K= \{ f_{\nu,0}:\nu \in U_{p^2}\}\cong U_{p^2} $ and $U_{p^2}$ is a cyclic group (see \cite[p.155]{gal}) of order $\varphi(p^2)$, where $\varphi$ is the Euler's phi-function, for each divisor $d$ of $ \varphi(p^2) $, there are exactly $\varphi(d) $ elements of order $d$ in $K$. Let $\nu\in U_{p^2}\setminus 1+J $. Then by  Lemma \ref{s3l1}, $f_{\nu,0} $ fixes exactly one element. If the order of $f_{\nu,0} $ is $t$ and if it divides $p-1$, then $f_{\nu, 0} $ is the product of $\frac{p^2-1}{t} $ disjoint $t$-cycles 
	                and  a cycle of length $1$. Assume that the order of $f_{\nu, 0} \in S_2$ is $tp$, where $t$ divides $p-1$. Then $\nu^{t}$ has order $p$ and so $\nu^{t} \in 1+J$. Thus by Lemma \ref{s3l2}, $ f_{\nu,0}^{t} = f_{\nu^{t},0}$ fixes each element of $J$. Hence $ f_{\nu, 0} $ is the product of  $\frac{p-1}{t} $ disjoint $tp$-cycles, $\frac{p-1}{t} $ disjoint $t$-cycles and a cycle of length $1$.\\

	                Let $f_{\nu, t^{\prime}} \in S_3$. Since $ \nu t^{\prime}\equiv t^{\prime} (\text{mod} ~ p^2)$ and the order of $\nu$ is $p$, thus the order of $f_{\nu,t^{\prime}}$ is $p$. Further,  by Lemma \ref{s3l2}, $f_{\nu,t^{\prime}} $ fixes each element of a coset of $J$ in $\mathbb{Z}_{p^2}$. Therefore $f_{\nu, t^{\prime}} $ is the product of $p-1$ disjoint $p$-cycles and $p$ disjoint cycles of length $1$.\\  
	                
	                Let $\nu \in 1+ J $ and $x \in \mathbb{Z}_{p^2} $. Then the congruence $(\nu-1)x \equiv - t^{\prime} (\text{mod} ~ p^2) $ has no solution if $t^{\prime} \in U_{p^2} $. Thus  if $ f_{\nu,t^{\prime}} \in S_4 $, then it does not fix any element of  $ \mathbb{Z}_{p^2} $. Therefore $ f_{\nu,t^{\prime}} $ 
	                 is a cycle of length $p^2$. 
	                 Hence the cyclic index of  $ \text{Aff}(1, p^{2}) $ is given by \begin{equation*} \begin{split}
	                 \text{P}_{\text{Aff}(1, p^{2})}(x_{1},\cdots ,x_{p^{2}})=&\frac{1}{p^{2}\varphi(p^{2})}[x_{1}^{p^{2}}+(p-1)x_{p}^{p} + p^{2}\sum_{t|p-1, t\neq 1}\varphi(t)x_{1}x_{t}^{\frac{p^{2}-1}{t}}+\\&
	                 p^{2} \sum_{t|p-1, t\neq 1}\varphi(tp)x_{1} x_{tp}^{\frac{p-1}{t}}x_{t}^{\frac{p-1}{t}}  + p (p-1)x_{1}^{p}x_{p}^{p-1}+ p\varphi(p^{2})x_{p^{2}}], \end{split}     
	                 \end{equation*}  where $ \varphi $ is the Euler's phi-function. This completes the proof of the Proposition. \end{proof}
                   
                    \begin{example}  By Proposition \ref{s3p1},
                    \begin{enumerate}
                    	\item   $\text{P}_{\text{Aff}(1, 9)}(2,\cdots ,2) = \frac{1}{54}[2^{9}+ 2^{4}+ 3^2\times 2^5 + 3^{2}\times 2^6 + 3\times 2^6 + 3^2\times 2^2]  =2\times 11 $   
                    	 \end{enumerate}
                     \begin{equation*} 
               \begin{split} ~~~~ 2.~   \text{P}_{\text{Aff}(1, 25)}(2,\cdots ,2) =& \frac{1}{500}[2^{25} + 2^{7} + 5^2\times 2^{13} + 5^2\times 2^8 + 5^2\times 2^7 + 5^2\times 2^6  + \\& 5\times 2^{11} + 5^2 \times 2^3] = 2\times 33,781. 
              \end{split} \end{equation*}
                 	Thus by Theorem \ref{s2t2},  $|\mathcal{I}tp(D_{18}, H)| = 11$ and 
	                	 $|\mathcal{I}tp(D_{50}, H) |= 33,781$ ,
                	where in each case, $H$ is a subgroup of order $2$.
	                	\end{example}

	           	  \end{document}